\documentclass[reqno,12pt,letterpaper]{amsart}
\usepackage{amsmath,amssymb,amsthm,graphicx,url,hyperref}

\newtheorem{theo}{Theorem}
\newtheorem{prop}{Proposition}[section]

\numberwithin{equation}{section}

\DeclareMathOperator{\Imag}{Im}
\DeclareMathOperator{\loc}{loc}
\DeclareMathOperator{\comp}{comp}
\DeclareMathOperator{\const}{const}
\DeclareMathOperator{\supp}{supp}
\DeclareMathOperator{\Div}{div}
\DeclareMathOperator{\Tr}{Tr}
\DeclareMathOperator{\Vol}{Vol}

\title[Exponential decay for Kerr--de Sitter]%
{Exponential energy decay\\
for Kerr--de Sitter black holes\\
beyond event horizons}
\author{Semyon Dyatlov}
\email{dyatlov@math.berkeley.edu}
\address{Department of Mathematics, Evans Hall, University of California,
Berkeley, CA 94720, USA}

\begin{document}

\begin{abstract}
We establish an exponential decay estimate for linear waves on the
Kerr--de Sitter slowly rotating black hole. Combining the cutoff
resolvent estimate of~\cite{skds} with the red-shift effect and a
parametrix near the event horizons, we obtain exponential decay on the
whole domain of outer communications.
\end{abstract}

\maketitle

We study decay of linear waves on the Kerr--de Sitter metric,
corresponding to a rotating black hole in a spacetime with positive
cosmological constant. (See for example~\cite{d-r} for the motivation
for the problem and a survey of recent results.)  Although in the
original coordinates $(t,r,\theta,\varphi)$ the metric is only defined
on $M=\{r_-<r<r_+\}$ and becomes singular on the event horizons
$\{r=r_\pm\}$, we will use a different coordinate system
$(t_+,r,\theta,\varphi_+)$, in which the metric can be extended beyond
the event horizons to $M_\delta=\{r_--\delta<r<r_++\delta\}$.  In
particular, $t_+\sim t+C_\pm \ln|r-r_\pm|$ near $r=r_\pm$, with
$C_\pm$ positive constants; see~\eqref{e:kerr-star} for precise
formulas and Section~1.2 in general for the description of the metric.
The speed of rotation of the black hole is described by the parameter
$a$; for $a=0$, we get the spherically symmetric Schwarzschild--de
Sitter metric. We establish the following
\begin{theo}\label{l:integrated-xpd}
Let $\Box_g$ be the d'Alembert--Beltrami operator of the Kerr--de
Sitter metric on $M_\delta$.  Fix $\varkappa>0$. For $a,\delta>0,\nu>0$ small enough
and any $s\geq 1$, there exists a constant $C$ such that if $u\in
H^{s+\varkappa+1}_{\loc}(M_\delta)$ is a solution to the equation
$\Box_g u=f\in H^{s+\varkappa}_{\loc}(M_\delta)$, with $\supp u\subset
\{t_+>-T\}$ for some $T$, then
\begin{equation}\label{e:main-estimate}
\|e^{\nu t_+}(u-\Pi_0f)\|_{H^s(M_\delta)}\leq
C \|e^{\nu t_+}f\|_{H^{s+\varkappa}(M_\delta)}.
\end{equation}
Here
$$
\Pi_0f={1+\alpha\over 4\pi(r_+^2+r_-^2+2a^2)}\int_M f\,d\Vol
$$
is a constant (note that we integrate over $M$, not the whole
$M_\delta$); $H^s$ norms are taken with respect to the
$(t_+,r,\theta,\varphi_+)$ coordinates (here $(\theta,\varphi_+)$ are treated
as spherical coordinates on $\mathbb S^2$).
\end{theo}
The main ingredient of the proof, which gives us exponential decay, is
the scattering resolvent constructed in~\cite{skds}. We modify the
argument of~\cite[Theorem~6]{skds} to get exponential decay for $u$ on
a certain compact subset $K_\delta\subset M$, under the condition that
$f$ is supported in $K_\delta$ as well
(Proposition~\ref{l:estimate-k-k}).  In the present paper, we specify
$\Box_g u$, rather than the Cauchy data of $u$, to facilitate the
proofs. However, it is not hard to convert
Theorem~\ref{l:integrated-xpd} to an exponential decay estimate for
the Cauchy problem, such as the one given in~\cite{skds}.

The second ingredient, described in Section~1.3, is the energy
estimate produced by red-shift effect at the event horizons, first
introduced by Dafermos and Rodnianski in~\cite{d03,d-r05,d-r08}.
The paper~\cite{d03} in particular introduced the idea that the red-shift
implies that boundedness and decay properties propagate from the event
horizon to a neighborhood of it in the black hole interior. The vector
field approach to the red-shift effect was introduced in~\cite{d-r05};
in~\cite{d-r08}, the method was extended to higher order estimates
using the remark that commutation generated further terms of favorable sign.
The paper~\cite{d-r05} established red-shift estimates for the Schwarzschild
black hole, while~\cite{d-r08} considered the case of slowly rotating Kerr;
the (subextremal) Kerr--de Sitter horizons are in particular covered
by~\cite[Theorems~7.1 and~7.2]{d-r}. It should also be noted that
in certain cases, such as extremal Reissner--Nordstr\"om spacetimes
considered by Aretakis~\cite{ar11,ar11-2}, the failure of the red-shift
is directly related to instabilities of linear waves at the event horizon.

In our presentation, we follow both~\cite[Section~3.3]{d-r} and the
paper~\cite{t-t} by Tataru and Tohaneanu on integrated decay for the
Kerr black hole.  Combining the red-shift effect with
Proposition~\ref{l:estimate-k-k}, we obtain an estimate on $u$ on the
whole $M_\delta$, provided that $f$ is still supported in
$K_\delta$. Finally, we use a Morawetz type argument together with
red-shift (Proposition~\ref{l:parametrix-estimate}) to construct an
exponentially decaying parametrix for the wave equation near the event
horizons and reduce the general problem to the case $\supp f\subset
K_\delta$.

Compared to the energy estimate for the Minkowski spacetime
(Proposition~\ref{l:minkowski-estimate}), we lose $1+\varkappa$
derivatives in Theorem~1, where $\varkappa>0$ can be arbitrarily
small.  This is related to the exponent in the polynomial resolvent
estimate of~\cite{skds}, which in turn is determined by the separation
of variables procedure employed there~\cite[Proposition~3.4]{skds}.
It is possible that a more careful analysis will yield a smaller loss
in derivatives; however, the presence of trapping indicates that loss
of regularity is inevitable (see~\cite{r} for a precise statement in
the now classical case of obstacle scattering).

Exponential decay of linear waves on the Schwarzschild--de Sitter
metric has been studied in~\cite{b-h,d-r2,m-sb-v}.  Dafermos and
Rodnianski~\cite{d-r2}, using vector field multipliers, proved that
linear waves decay faster than every negative power of $t_+$.  Bony
and H\"afner~\cite{b-h}, building on earlier work on the scattering
resolvent by S\'a~Barreto and Zworski~\cite{sb-z}, showed exponential
decay away from the event horizons. Finally, Melrose, S\'a~Barreto,
and Vasy~\cite{m-sb-v} proved exponential decay up to the event
horizons.  The latter result, combined with the recent work on
normally hyperbolic trapping~\cite{w-z} and gluing semiclassical
resolvent estimates~\cite{d-v}, can be applied to certain short-range
stationary perturbations of the Schwarzschild--de Sitter spacetime;
see~\cite[Corollary~6.1]{d-v}.  It should be noted, however, that
Kerr--de Sitter is not an acceptable perturbation, in particular
because the theorem of Mazzeo and Melrose~\cite{m-m} does not apply to
the low energy situation anymore. Therefore, at the moment, the
results of~\cite{skds} seem necessary for obtaining exponential decay
of waves on Kerr--de Sitter.

\section{Kerr--de Sitter metric and the red-shift effect}

\subsection{Energy estimates}

We recall some well-known facts from Lorentzian geometry; see for
example~\cite[Appendices]{d-r} or~\cite[Section~2.8]{tay} for a more
detailed account.

Let $M$ be an $n$-dimensional smooth manifold and $g$ be a Lorentzian
metric; that is, a symmetric $(0,2)$-tensor $g$ of signature
$(1,n-1)$. (Sometimes a different convention is used, in which the
metric has signature $(n-1,1)$.)  The basic example is the space
$\mathbb R_{t,x}^n$ with the Minkowski metric
$$
dt^2-\sum_{j=1}^{n-1} dx_j^2.
$$
A tangent vector $X$ is called timelike if $g(X,X)>0$, null if
$g(X,X)=0$, and spacelike if $g(X,X)>0$. If $X$ and $Y$ are two
timelike vectors, then we say that they point in the same direction if
$g(X,Y)>0$ and they point in opposite directions if $g(X,Y)<0$. This
definition can be extended to cases when $X$ and/or $Y$ is a nonzero
null vector. A hyperplane in a tangent space is called spacelike if
its normal vector (with respect to $g$) is timelike, timelike
if its normal vector is spacelike, and null if its
normal vector is null. Note that the restriction of $g$ to a spacelike
hyperplane is negative definite, while the restriction to a null plane
has signature $(0,n-2)$.

We now describe a way of obtaining energy estimates for the wave
equation on Lorentzian manifolds. Let $\Omega\subset M$ be a bounded
domain and $u\in C^\infty(\bar\Omega)$. Define the symmetric
$(0,2)$-tensor $T_{\nabla u}$ by the formula
\begin{equation}\label{e:stress-energy}
T_{\nabla u}(X,Y)=(Xu)(Yu)-{1\over 2}g(\nabla u,\nabla u)g(X,Y),
\end{equation}
valid for all vector fields $X,Y$ on $\Omega$. Note that for fixed $X$
and $Y$, $T_{\nabla u}(X,Y)$ is a quadratic form in $\nabla u$.  If
$X$ and $Y$ are both timelike, then this form is positive definite in
$\nabla u$ for $X$ and $Y$ pointing in the same direction and negative
definite otherwise. Same is true if $X$ and/or $Y$ is null, with the
form being nonnegative or nonpositive, respectively.

Fix a vector field $X$ on $\Omega$ and consider the vector
field $J_X(u)$, given by the formula
$$
g(J_X(u),Y)=T_{\nabla u}(X,Y),
$$
valid for all vector fields $Y$. The divergence theorem then gives
\begin{equation}\label{e:divergence-theorem}
\int_{\partial\Omega} T_{\nabla u}(X,\vec n)\,dS
=\int_{\Omega}\Div J_X(u)\,d\Vol.
\end{equation}
Here $\vec n$ is the unit normal vector pointing outward (in the sense
that $g(\vec n,Z)>0$ for every vector $Z$ pointing outside of
$\Omega$); $dS$ is the area measure induced by the restriction of $g$
to $\partial\Omega$, and $d\Vol$ is the volume measure induced by
$g$. One has to take care when defining the left-hand side
of~\eqref{e:divergence-theorem} at the points where $\partial\Omega$
is null, as $\vec n$ blows up, being both unit and null, and $dS$ is
equal to zero; see~\cite[Appendix~C]{d-r} for details.  The discussion
following~\eqref{e:stress-energy} implies
\begin{prop}\label{l:current-positive-integral}
Let $\mathcal C$ be an open subset of $\partial\Omega$ whose tangent
space is either spacelike or null at every point. Moreover, assume
that $X$ is timelike and points outside of $\Omega$ on $\mathcal
C$. Then for every $u$,
\begin{equation}\label{e:positive-flux}
\int_{\mathcal C}T_{\nabla u}(X,\vec n)\,dS\geq 0.
\end{equation}
\end{prop}
The sign of the flux of $J_X$ over a timelike piece of
$\partial\Omega$ cannot be determined in general; however, we can find
it if $u$ satisfies a boundary condition:
\begin{prop}\label{l:morawetz-positive-integral}
Let $\mathcal C$ be an open timelike subset of $\partial\Omega$ and
assume that $u|_{\mathcal C}=0$. If $X$ points inside of $\Omega$ on
$\mathcal C$, then~\eqref{e:positive-flux} holds.
\end{prop}
\begin{proof}
We have $\nabla u=v\vec n$ on $\mathcal C$, for some function $v$.
Then
$$
T_{\nabla u}(X,\vec n)={v^2\over 2}g(X,\vec n)g(\vec n,\vec n)
=-{v^2\over 2}g(X,\vec n)\geq 0.\qedhere
$$
\end{proof}

Finally, we relate the divergence of $J_X$ to the d'Alembert--Beltrami
operator $\Box_g$:
\begin{prop}\label{l:energy-divergence}
Let $\mathcal L_X g$ be the Lie derivative of $g$ with respect to $X$,
and consider the symmetric $(0,2)$-tensor $K^X$ given by
$$
K^X={1\over 2}\mathcal L_X g-{1\over 4}\Tr(g^{-1}\mathcal L_X g)g.
$$
Then
$$
\Div J_X(u)=(Xu)\Box_g u+K^X(\nabla u,\nabla u).
$$
\end{prop}
As a basic application, we prove the energy estimate for the
constant-coefficient wave equation:
\begin{prop}\label{l:minkowski-estimate}
Fix $0<T<R$ and consider the domain
$$
\Omega=\{0<t< T,\ |x|<R-t\}
$$
in the Minkowski spacetime.  Let $u\in C^\infty(\bar\Omega)$ and
define the energy
$$
E(s)={1\over 2}\int_{t=s\atop |x|<R-s} |u_t|^2+|\partial_x u|^2\,dx,\
0\leq s\leq T.
$$
Then
$$
E(T)\leq E(0)+\int_{\Omega} u_t \Box u\,dtdx.
$$
\end{prop}
\begin{proof}
We take $X=\partial_t$ and apply~\eqref{e:divergence-theorem} on
$\Omega$. Since $X$ is Killing, $K^X=0$ and thus
$$
\int_{\partial\Omega} T_{\nabla u}(X,\vec n)\,dS=\int_\Omega u_t\Box u\,dtdx.
$$
Now, the boundary of $\Omega$ consists of the following pieces:
$$
\begin{gathered}
\mathcal P_0=\{t=0,\ |x|<R\},\\
\mathcal P_T=\{t=T,\ |x|<R-T\},\\
\mathcal C=\{0\leq t\leq T,\ |x|=R-t\}.
\end{gathered}
$$
The integral over $\mathcal P_0$ is equal to $-E(0)$ and the integral
over $\mathcal P_T$ is equal to $E(T)$.  Finally, the integral over
$\mathcal C$ is nonnegative by
Proposition~\ref{l:current-positive-integral}, as $\mathcal C$ is null
and $\partial_t$ points outside of $\Omega$ on $\mathcal C$.
\end{proof}

\subsection{Kerr--de Sitter metric}

The Kerr--de Sitter metric is given by
$$
\begin{gathered}
g=-\rho^2\Big({dr^2\over \Delta_r}+{d\theta^2\over\Delta_\theta}\Big)\\
-{\Delta_\theta\sin^2\theta\over (1+\alpha)^2\rho^2}
(a\,dt-(r^2+a^2)\,d\varphi)^2\\
+{\Delta_r\over (1+\alpha)^2\rho^2}
(dt-a\sin^2\theta\,d\varphi)^2.
\end{gathered}
$$
Here $M_0$ is the mass of the black hole, $\Lambda$ is the cosmological
constant (both of which we assume to be fixed),
and $aM_0$ is the angular momentum (which we assume to be small);
$$
\begin{gathered}
\Delta_r=(r^2+a^2)\Big(1-{\Lambda r^2\over 3}\Big)-2M_0r,\\
\Delta_\theta=1+\alpha\cos^2\theta,\\
\rho^2=r^2+a^2\cos^2\theta,\
\alpha={\Lambda a^2\over 3}.
\end{gathered}
$$
The metric in the $(t,r,\theta,\varphi)$ coordinates is defined for
$\Delta_r>0$; we assume that $r_\pm$ are two positive roots of
the equation $\Delta_r=0$, such that $\Delta_r>0$ on
the open interval $0<r_-<r<r_+<\infty$. The variables $\theta\in [0,\pi]$ and
$\varphi\in \mathbb R/2\pi \mathbb Z$ are the spherical coordinates on
the sphere $\mathbb S^2$. The spacetime is then
$$
M=\mathbb R_t\times (r_-,r_+)\times \mathbb S^2_{\theta,\varphi}.
$$
(Note the difference in notation with~\cite{skds}.) The volume form is
$$
d\Vol={\rho^2\sin\theta\over (1+\alpha)^2}\,dtdrd\theta d\varphi.
$$
For $a=0$, we get the Schwarzschild--de Sitter metric:
$$
\begin{gathered}
g_0=-{r^2\over\Delta_r}dr^2
+{\Delta_r\over r^2}dt^2
-r^2 g_S,
\end{gathered}
$$
where
$$
g_S=d\theta^2+\sin^2\theta d\varphi^2
$$
is the round metric on the sphere of radius 1.

Next, we introduce a modification of the Kerr-star coordinates
(see~\cite[Section~5.1]{d-r}).  We remove the singularities at
$r=r_\pm$ by making the change of variables $(t,r,\theta,\varphi)\to
(t_+,r,\theta,\varphi_+)$, where
\begin{equation}\label{e:kerr-star}
t_+=t-F_t(r),\
\varphi_+=\varphi-F_\varphi(r).
\end{equation}
Note that $\partial_{t_+}=\partial_t$ and
$\partial_{\varphi_+}=\partial_\varphi$.  The functions $F_t$ and
$F_\varphi$ are required to be smooth on $(r_-,r_+)$ and satisfy the
following condition:
$$
F'_t(r)=\pm{(1+\alpha)(r^2+a^2)\over\Delta_r},\
F'_\varphi(r)=\pm{(1+\alpha)a\over\Delta_r}\text{ for }
|r-r_\pm|<\varepsilon.
$$
Here $\varepsilon>0$ is some fixed small constant. The metric $g$ in
the $(t_+,r,\theta,\varphi_+)$ coordinates takes the following form
for $|r-r_\pm|<\delta$:
$$
\begin{gathered}
-\rho^2{d\theta^2\over\Delta_\theta}
-{\Delta_\theta\sin^2\theta\over (1+\alpha)^2\rho^2}
(a\,dt_+-(r^2+a^2)\,d\varphi_+)^2\\
+{\Delta_r\over (1+\alpha)^2\rho^2}(dt_+-a\sin^2\theta\,d\varphi_+)^2
\pm {2\over (1+\alpha)}(dt_+-a\sin^2\theta\,d\varphi_+)dr.
\end{gathered}
$$
We see that the metric is smooth up to the event horizons $\{r=r_\pm\}$;
moreover, for $\delta$ small enough, we can extend it to
$$
M_\delta=\{r_--\delta\leq r\leq r_++\delta\}.
$$
The event horizons are null, while the surfaces $\{r=r_0\}$ are
spacelike for $r_0\not\in [r_-,r_+]$. The time surfaces
$\{t_+=\const\}$ are null near the event horizons; however, one can
shift the time variable a little bit (see~\cite[Section~1]{skds}) to
make the problem
$$
\Box_g u=f\in C_0^\infty(M_\delta),\
\supp u\subset \{t_+> -T\}\text{ for some }T
$$
well-posed. We call $u$ the forward solution of the equation $\Box_g u=f$.

Finally, note that the field $\partial_t$ (which is the same in the
$(t,r,\theta,\varphi)$ and $(t_+,r,\theta,\varphi_+)$ coordinates) is
not timelike on $M$ inside the two surfaces located $O(a)$-close (in the $r$ variable)
to the event horizons; these surfaces are called the ergospheres.

\subsection{Red-shift effect}

In this section, we prove the following energy estimate:
\begin{prop}\label{l:red-shift}
For $\delta>0$, define
$$
K_\delta=\{r_-+\delta<r<r_+-\delta\}\subset M.
$$
Then for $\delta$, $a$, and $\nu>0$ small enough, $s$ a nonnegative
integer, and every forward solution $u$ to the equation $\Box_g u=f\in
C_0^\infty(M_\delta)$, we have%
\footnote{We write $(A)\lesssim (B)$, if
there exists some constant $C$, independent of the choice of $f$, such
that $A\leq C(B)$. Here $C$ might depend on the parameters of the
problem such as $\nu$, $s$, and $\varkappa$.}
\begin{equation}\label{e:red-shift}
\|e^{\nu t_+} u\|_{H^{s+1}(M_\delta)}
\lesssim \|e^{\nu t_+} f\|_{H^s(M_\delta)}
+\|e^{\nu t} u\|_{H^{s+1}(K_\delta)}.
\end{equation}
\end{prop}
We start the proof with the construction of a special vector field;
see also~\cite[Proposition~3.3.1]{d-r}.
\begin{prop}\label{l:red-shift-X}
For $\delta>0$ and $a$ small enough, there exists a vector field $X$
defined on $M_\delta\setminus K_{2\delta}$, with the following
properties:
\begin{itemize}
\item $X$ is stationary; that is, $[X,\partial_t]=0$.
\item $X$ is timelike and $Xt_+>0$, $\pm Xr>0$ on $M_\delta\setminus K_{2\delta}$.
\item The tensor $K^X$, defined in Proposition~\ref{l:energy-divergence},
is negative definite on $M_\delta\setminus K_{2\delta}$.
\end{itemize}
\end{prop}
\begin{proof}
We will construct $X$ for $a=0$; same field will work for small $a$ since
the components of the Kerr--de Sitter metric near the event horizons are
continuous functions of $a$. Moreover, since $\delta>0$ can be chosen arbitrarily
small, we only need to verify properties of $X$ at the event horizons.
We use the $(t_+,r,\theta,\varphi_+)$ coordinates.
The metric for $a=0$ has the form
$$
{\Delta_r\over r^2}\,dt_+^2\pm 2dt_+dr-r^2g_S
$$
for $|r-r_\pm|<2\delta$; if we take
$$
X=X_r(r)\partial_r+X_t(r) \partial_t,
$$
where $X_r,X_t$ are some functions, then at $r=r_\pm$,
$$
\begin{gathered}
\mathcal L_X g=
X_r\left(
{\Delta'_r\over r^2}dt_+^2-2rg_S\right)
\pm 2[\partial_rX_t dr^2+
\partial_r X_rdrdt_+],\\
K^X={X_r\Delta'_r\over 2r^2}dt_+^2\pm \partial_rX_t dr^2
\mp {2X_r\over r}drdt_++{r^2\over 2}\partial_rX_r g_S.
\end{gathered}
$$
We put $X_t=1$ and $X_r=\pm 1$ at $r=r_\pm$; then the field $X$ is
timelike and $dt_+(X)>0$.  To make $K^X$ negative definite, it then
suffices to take $\mp \partial_r X_t$ positive and large enough and
$\partial_r X_r$ negative at the event horizons.
\end{proof}
\noindent\textbf{Remark.} Note that the only component
of $K^X$ whose sign is definite independently of the choice of
$\partial_r X$ is
$$
K^X(\partial_t,\partial_t)={1\over 2}\mathcal (L_x g)_{t_+t_+}
=-g(X,\nabla_{\partial_t}\partial_t).
$$
One can compute
\begin{equation}\label{e:classical-red-shift}
\nabla_{\partial_t}\partial_t=\kappa \partial_t
\end{equation}
for some constant $\kappa>0$; then,
$$
K^X(\partial_t,\partial_t)=\mp\kappa X_r,\
r=r_\pm.
$$
The equation~\eqref{e:classical-red-shift} can be interpreted as
follows: the momentum is exponentially decaying as a function of the
geodesic parameter on the family of trapped geodesics $\{r=r_\pm,\
(\theta,\varphi)=\const\}$. This is related to the classical red-shift
effect; see~\cite[Sections~3.3.2 and~7.1]{d-r} for more details. 

We are now ready to prove Proposition~\ref{l:red-shift}. To facilitate
the inductive argument for estimating higher derivatives, we show
the following more general fact:
\begin{prop}\label{l:redshift-intermediate}
Assume that $\psi(r)$ is a function on $M_\delta$ such that $\psi\geq
0$ outside of $K_\delta$, and $u$ is a forward solution to the
equation
$$
(\Box_g +\psi X)u=f\in C_0^\infty(M_\delta).
$$
Here $X$ is the field constructed in Proposition~\ref{l:red-shift-X}.
Then for $a,\delta>0,\nu>0$ small enough and each nonnegative integer $s$,
\begin{equation}\label{e:intermediate-redshift}
\|e^{\nu t_+} u\|_{H^{s+1}(M_\delta)}
\lesssim \|e^{\nu t_+} f\|_{H^s(M_\delta)}
+\|e^{\nu t_+} u\|_{H^{s+1}(K_\delta)}.
\end{equation}
\end{prop} 
\begin{proof} 
We use induction on $s$. First, assume that $s=0$. Take a nonnegative
function $\chi(r)$ on $M_\delta$ such that $\chi=0$ near
$K_{2\delta}$, but $\chi=1$ away from $K_\delta$. Let $T>0$ and apply
the divergence theorem in the region
$$
\Omega_T=M_\delta\cap \{t_+<T\}
$$
to the vector field
$$
V=e^{2\nu t_+}\chi J_X(u).
$$
Here $J_X$ is defined in Section~1.1.  (The divergence theorem holds,
despite $\Omega_T$ being noncompact, since $u$ is a forward solution.) 
We compute by Proposition~\ref{l:energy-divergence}
$$
\begin{gathered}
\Div V=e^{2\nu t_+}[2\nu \chi dt_+(J_X(u))+e^{2\nu t_+}\chi' dr(J_X(u))\\
+\chi(Xu)f-\chi\psi (Xu)^2+\chi K^X(\nabla u,\nabla u)].
\end{gathered}
$$
The flux of $V$ is nonnegative by
Proposition~\ref{l:current-positive-integral}; therefore, by letting
$T\to +\infty$ we get
$$
\begin{gathered}
\int_{M_\delta\setminus K_\delta} -e^{2\nu t_+}K^X(\nabla u,\nabla u)\,d\Vol
\lesssim \nu \|e^{\nu t_+}u\|_{H^1(M_\delta)}^2\\
+\|e^{\nu t_+}u\|_{H^1(K_\delta)}^2
+\|e^{\nu t_+}u\|_{H^1(M_\delta)}\cdot \|f\|_{L^2(M_\delta)}
\end{gathered}
$$
Since $K^X$ is negative definite on $M_\delta\setminus K_\delta$
and by Poincar\'e inequality, we have for $\nu$ small enough,
$$
\begin{gathered}
\|e^{\nu t_+}u\|_{H^1(M_\delta)}^2
\lesssim
\|e^{\nu t_+}u\|_{H^1(K_\delta)}^2
+\|e^{\nu t_+}u\|_{H^1(M_\delta)}\cdot \|f\|_{L^2(M_\delta)}
\end{gathered}
$$
This finishes the proof of~\eqref{e:intermediate-redshift} for $s=0$.

Now, assume that $s\geq 1$ and~\eqref{e:intermediate-redshift} is true
for $s-1$; we will prove it for $s$ following~\cite[Sections~1.7.5 and~10]{d-r08}
(see also~\cite[Theorem~4.4]{t-t}).
First, let $Y$ be equal to either $\partial_t$ or a Killing field on
$\mathbb S^2$; then $[\psi X,Y]=0$ and, since the metric is
spherically symmetric for $a=0$, $[\Box_g,Y]$ is a second order
differential operator with $O(a)$ coefficients.  We have
$$
(\Box_g+\psi X)Yu=Yf+[\Box_g,Y]u;
$$
therefore, by~\eqref{e:intermediate-redshift},
$$
\begin{gathered}
\|e^{\nu t_+}Yu\|_{H^s(M_\delta)}\lesssim
\|e^{\nu t_+}Yf\|_{H^{s-1}(M_\delta)}
+O(a)\|e^{\nu t_+}u\|_{H^{s+1}(M_\delta)}
+\|e^{\nu t_+}Yu\|_{H^s(K_\delta)}.
\end{gathered}
$$
Therefore, if $\hat \partial u$ is composed of derivatives of $u$ with
respect to $t_+,\theta,\varphi_+$, then
\begin{equation}\label{e:red-shift-i1}
\|e^{\nu t_+}\hat \partial u\|_{H^s(M_\delta)}\lesssim
\|e^{\nu t_+}f\|_{H^s(M_\delta)}
+O(a)\|e^{\nu t_+}u\|_{H^{s+1}(M_\delta)}
+\|e^{\nu t_+}u\|_{H^{s+1}(K_\delta)}.
\end{equation}
Now, we estimate $\partial_r u$. We can write
$$
[\Box_g+\psi X,\partial_r]=-\eta X \partial_r+L,
$$
where $L$ is a second order differential operator not containing any
$\partial_r^2$ terms and $\eta$ is positive near the event
horizons. Then
$$
\|e^{\nu t_+}Lu\|_{H^{s-1}(M_\delta)}\lesssim
\|e^{\nu t_+} \hat \partial u\|_{H^s(M_\delta)}+
\|e^{\nu t_+} u\|_{H^s(M_\delta)}.
$$
We have
$$
(\Box_g+(\psi+\eta)X)\partial_ru=\partial_r f+Lu;
$$
since $\psi+\eta\geq 0$ near the event horizons,
by~\eqref{e:intermediate-redshift} applied to $\partial_r u$
and~\eqref{e:red-shift-i1} we get
$$
\begin{gathered}
\|e^{\nu t_+}u\|_{H^{s+1}(M_\delta)}\lesssim
\|e^{\nu t_+}\partial_r u\|_{H^s(M_\delta)}
+\|e^{\nu t_+}\hat \partial u\|_{H^s(M_\delta)}\\\lesssim
\|e^{\nu t_+}f\|_{H^s(M_\delta)}
+\|e^{\nu t_+}u\|_{H^{s+1}(K_\delta)}
+\|e^{\nu t_+}\hat \partial u\|_{H^s(M_\delta)}
+\|e^{\nu t_+}u\|_{H^s(M_\delta)}\\\lesssim
\|e^{\nu t_+}f\|_{H^s(M_\delta)}
+\|e^{\nu t_+}u\|_{H^{s+1}(K_\delta)}
+O(a)\|e^{\nu t_+}u\|_{H^{s+1}(M_\delta)};
\end{gathered}
$$
it remains to take $a$ small enough.
\end{proof}

\section{Proof of exponential decay}

Throughout this section, $u$ is a forward solution to the equation
$\Box_g u=f$, with $f\in C_0^\infty(M_\delta)$. (The estimates for general
$f$ can then be obtained by a density argument.) 

\subsection{Case of $f$ supported in $K_\delta$}
First of all, we use
the resolvent estimates of~\cite{skds} to obtain exponential decay
away from the event horizons:
\begin{prop}\label{l:estimate-k-k}
Fix $\delta>0,\varkappa>0$ and assume that $\chi(r)\in
C_0^\infty(r_-+\delta,r_+-\delta)$.  Then for $a$ small enough and
every $s\geq 0$, we have
$$
\|e^{\nu t}\chi(r) (u-\Pi_0f)\|_{H^s}\lesssim
\|e^{\nu t} f\|_{H^{s+\varkappa}}
$$
for every $f\in C_0^\infty(K_\delta)$. (We can use $e^{\nu t}$ in place
of $e^{\nu t_+}$, as $|t-t_+|$ is bounded and the two weights are equivalent
in $K_\delta$.) 
\end{prop}
\begin{proof}
We use the argument of~\cite[Theorem~6]{skds}.
By~\cite[Proposition~1.1]{skds}, $e^{-Ct}u$ is tempered in the time
variable for some constant $C$; therefore, the Fourier--Laplace
transform
$$
\hat u(\omega,\cdot)=\int_{-\infty}^\infty e^{it\omega} u(t,\cdot)\,dt
$$
is well-defined and holomorphic in $\{\Imag\omega>C\}$. Let
$K_S=(r_-+\delta,r_+-\delta)\times \mathbb S^2$ be the space slice of
$K_\delta$.  We choose $a$ small enough so that~\cite[Theorem~2]{skds}
provides us with the scattering resolvent $R_g(\omega):L^2(K_S)\to
L^2(K_S)$; it is a family of operators meromorphic in the entire
complex plane.  By~\cite[Proposition~1.2]{skds},
$$
\chi(r)\hat u(\omega)=\chi(r)R_g(\omega)(\rho^2\hat f(\omega)),
$$
where $\rho(r,\theta)$ is the smooth function defined in Section~1.2
and $\hat f(\omega)$ is an entire function that is rapidly decaying in
$\omega$ for $\Imag\omega$ bounded, with values in
$C_0^\infty(K_S)$. Now, there exists $\nu>0$ such that $R_g(\omega)$
is holomorphic and polynomially bounded in $\{\Imag\omega\geq -\nu\}$,
except for a pole at zero~\cite[Theorems~4 and~5]{skds}. Therefore, we
can use Fourier inversion formula and contour deformation to get
\begin{equation}\label{e:fourier-inversion}
\chi(r)(u(t,\cdot)-\Pi_0f)
={1\over 2\pi}\int_{\Imag\omega=-\nu}
e^{-it\omega}\chi(r)R_g(\omega)(\rho^2\hat f(\omega))\,d\omega,
\end{equation}
the residue at zero being exactly $\Pi_0 f$.  Now, let $s\in \mathbb
R$, put $h=\langle\omega\rangle^{-1}$, and introduce the semiclassical
Sobolev space $H^s_{h,\comp}(K_S)\subset \mathcal E'(K_S)$; the norm
of of $v\in \mathcal E'(K_S)$ in this space is given by $\|\langle
hD\rangle^sv\|_{L^2}$, where $\langle hD\rangle ^s$ is a Fourier
multiplier and $v$ is extended by zero to $\mathbb R^3\supset K_S$.
Then the norm of $e^{\nu t}f$ in $H^s(K_\delta)$ is equivalent to the
norm of $\langle\omega\rangle^s\hat f$ in $L^2_\nu H^s_h$, where
$$
\|v\|_{L^2_\nu H^s_h}^2
=\int_{\Imag\omega=-\nu} \|v(\omega)\|_{H^s_h}^2\,d\omega.
$$
This, together with the resolvent estimate of~\cite[Theorem~5]{skds} gives
for $\varkappa$ fixed and $\nu$ small enough,
$$
\|\langle\omega\rangle^s\chi(r)R_g(\omega)(\rho^2\hat f(\omega))\|_{L^2_\nu H^0_h}
\lesssim \|e^{\nu t}f\|_{H^{s+\varkappa}}.
$$
Now, we use that $R_g(\omega)$ is a right inverse to the second order
differential operator~\cite[Section~1]{skds}
$$
\begin{gathered}
P_g(\omega)=D_r(\Delta_r D_r)-{(1+\alpha)^2\over\Delta_r}((r^2+a^2)\omega-aD_\varphi)^2\\
+{1\over\sin\theta}D_\theta(\Delta_\theta\sin\theta D_\theta)
+{(1+\alpha)^2\over\Delta_\theta\sin^2\theta}(a\omega\sin^2\theta-D_\varphi)^2;
\end{gathered} 
$$
then $h^2
P_g(\omega)$ is a semiclassical pseudodifferential operator and for
$a$ small enough, it is elliptic on $K_S$ outside of some $\omega$-independent
compact set. (This is equivalent to saying that $K_\delta$ does not
intersect the ergosphere.) We can construct a semiclassical parametrix
(see for example~\cite[Section~4.5]{e-z}
or~\cite[Proposition~5.1]{skds}); i.e., a properly supported
semiclassical pseudodifferential operator on $Q$ on $K_S$ that maps
$H^s_{h,\loc}(K_S)\to H^{s+2}_{h,\loc}(K_S)$ for all $s$ and such that
$I-Qh^2P_g(\omega)$ maps $H^{-N}_{h,\loc}(K_S)\to H^{N}_{h,\loc}(K_S)$
with norm $O(1)$ for all $N$.  Then for any $\chi_1(r)\in
C_0^\infty(r_-,r_+)$ that is nonzero near $\supp\chi$ and any $v\in
C_0^\infty(K_S)$, we can apply $I-Qh^2P_g(\omega)$ to $R_g(\omega)v$
to get
\begin{equation}\label{e:semiclassical-elliptic}
\|\chi(r) R_g(\omega) v\|_{H^s_h}\lesssim
\langle\omega\rangle^{-2}\|v\|_{H^{s-2}_h}+\|\chi_1(r)R_g(\omega)v\|_{H^0_h}.
\end{equation}
Therefore,
$$
\|\langle\omega\rangle^s\chi(r)R_g(\omega)(\rho^2\hat f(\omega))\|_{L^2_\nu H^s_h}
\lesssim \|e^{\nu t}f\|_{H^{s+\varkappa}};
$$
it remains to combine this with~\eqref{e:fourier-inversion}.
\end{proof}
Combining the above fact with the red-shift estimate, we get
\begin{prop}\label{l:estimate-k-m}
Fix $\delta>0$ such that Proposition~\ref{l:red-shift} holds and
choose $a$ small enough so that Proposition~\ref{l:estimate-k-k} holds
for $\delta/2$ in place of $\delta$. Take $\nu>0$ small enough so that
both propositions above hold.  Then for $s\geq 1$ and every
$\varkappa>0$, we have
$$
\|e^{\nu t_+}(u-\Pi_0 f)\|_{H^s(M_\delta)}\lesssim
\|e^{\nu t}f\|_{H^{s+\varkappa}(K_\delta)},
$$
for every $f\in C_0^\infty(K_\delta)$.
\end{prop}
\begin{proof}
We consider the case of integer $s$; the general case follows
by interpolation in Sobolev spaces (see for example~\cite[Section~4.2]{tay}).
Let $\psi(t_+)$ be a smooth function that is equal to 1 for $t_+$
large positive and to 0 for $t_+$ large negative; take large $T\in
\mathbb R$ and apply Proposition~\ref{l:red-shift} to
$u-\psi(t_++T)\Pi_0f$:
$$
\begin{gathered}
\|e^{\nu t_+}(u-\Pi_0f)\|_{H^s(M_\delta)}
\\\lesssim \|e^{\nu t_+}(u-\psi(t_++T)\Pi_0 f)\|_{H^s(M_\delta)}
+\|e^{\nu t_+}(1-\psi(t_++T))\Pi_0 f\|_{H^s(M_\delta)}\\\lesssim
\|e^{\nu t_+}f\|_{H^{s-1}(M_\delta)}
+\|e^{\nu t_+}(u-\Pi_0f)\|_{H^s(K_\delta)}
+e^{-\nu T}|\Pi_0 f|;
\end{gathered}
$$
the second term above is estimated by Proposition~\ref{l:estimate-k-k}
and the third one tends to zero at $T\to +\infty$.
\end{proof}

\subsection{General case}
The idea is to construct an exponentially decaying function $u_1$
solving the equation $\Box_g u_1=f$ near the event horizons and then
estimate the difference $u-u_1$ by Proposition~\ref{l:estimate-k-m}.
We let $u_1\in C^\infty(M_\delta\setminus K_{2\delta})$
solve the following initial/boundary value problem:
$$
\begin{gathered}
\Box_g u_1=f\text{ in }M_\delta\setminus K_{2\delta},\\
\supp u_1\subset\{t_+>-T\}\text{ for some }T,\\
u_1|_{\partial K_{2\delta}}=0.
\end{gathered}
$$
Note that the surfaces $\partial K_{2\delta}=\{r=r_\pm\mp 2\delta\}$
are timelike; therefore, this problem has a unique solution (see for
example~\cite[Theorem~24.1.1]{ho3}. This solution is exponentially
decaying in time:
\begin{prop}\label{l:parametrix-estimate}
For $\delta>0,\nu>0$ small enough, $a$ small enough depending on $\delta$,
and every $s\geq 0$,
\begin{equation}\label{e:parametrix-estimate}
\|e^{\nu t_+}u_1\|_{H^{s+1}(M_\delta\setminus K_{2\delta})}\lesssim
\|e^{\nu t_+}f\|_{H^s(M_\delta\setminus K_{2\delta})}.
\end{equation}
\end{prop}
\begin{proof}
First, consider the case $s=0$. We argue as in the proof of
Proposition~\ref{l:redshift-intermediate}, using the vector field $X$
constructed in Proposition~\ref{l:red-shift-X}.  Namely, we apply the
divergence theorem to the vector field $V=e^{2\nu t_+}J_X(u_1)$ in the
region
$$
\Omega_T=(M_\delta\setminus K_{2\delta})\cap \{t_+\leq T\}.
$$
The flux of $V$ over $\{t_+=T\}$ and $\partial M_\delta$ is
nonnegative by Proposition~\ref{l:current-positive-integral}; the flux
over $\partial K_{2\delta}$ is nonnegative by
Proposition~\ref{l:morawetz-positive-integral}.  Computing the
divergence of $V$ by Proposition~\ref{l:energy-divergence}, we
get~\eqref{e:parametrix-estimate}.

Now, we assume that~\eqref{e:parametrix-estimate} is true for $s-2$
and prove it for $s$; the rest follows by induction and interpolation
in Sobolev spaces. For $a$ small enough, $\partial_t$
is timelike in $K_\delta\setminus K_{2\delta}$; therefore, for
large enough constant $C_0$, the operator
$$
L=C_0 \partial_t^2-\Box_g
$$
is elliptic on $K_\delta\setminus K_{2\delta}$. Since $u_1$ satisfies
the Dirichlet boundary condition on $\partial K_{2\delta}$, we have
$$
\begin{gathered}
\|e^{\nu t_+}u_1\|_{H^{s+1}(K_\delta\setminus K_{2\delta})}\lesssim
\|e^{\nu t_+}Lu_1\|_{H^{s-1}(M_\delta\setminus K_{2\delta})}
+\|e^{\nu t_+}u_1\|_{H^{s-1}(M_\delta\setminus K_{2\delta})}\\\lesssim
\|e^{\nu t_+}\partial_t^2u_1\|_{H^{s-1}(M_\delta\setminus K_{2\delta})}
+\|e^{\nu t_+}f\|_{H^{s-1}(M_\delta\setminus K_{2\delta})}\\\lesssim
\|e^{\nu t_+}\partial_t^2 f\|_{H^{s-2}(M_\delta\setminus K_{2\delta})}
+\|e^{\nu t_+}f\|_{H^{s-1}(M_\delta\setminus K_{2\delta})}\lesssim
\|e^{\nu t_+}f\|_{H^s(M_\delta\setminus K_{2\delta})}.
\end{gathered}  
$$
Here we applied ~\eqref{e:parametrix-estimate} to $u_1$ and $\partial_t^2u_1$
and used that $\Box_g$ commutes with $\partial_t^2$. 

Now, take a nonnegative function $\chi_\delta(r)\in C^\infty$ such that
$\chi_\delta=0$ near $K_{2\delta}$, but $\chi_\delta=1$ away from $K_\delta$.
We can use the above estimate and apply
Proposition~\ref{l:red-shift} to $\chi_\delta u_1$ to get
$$
\begin{gathered}
\|e^{\nu t_+}u_1\|_{H^{s+1}(M_\delta\setminus K_{2\delta})}\lesssim
\|e^{\nu t_+}u_1\|_{H^{s+1}(K_\delta\setminus K_{2\delta})}
+\|e^{\nu t_+}\chi_\delta u_1\|_{H^{s+1}(M_\delta)}\\\lesssim
\|e^{\nu t_+}f\|_{H^s(M_\delta\setminus K_{2\delta})}
+\|e^{\nu t_+}\chi_\delta u_1\|_{H^{s+1}(K_\delta)}
+\|e^{\nu t_+}[\Box_g,\chi_\delta]u_1\|_{H^s(K_\delta)}\\\lesssim
\|e^{\nu t_+}f\|_{H^s(M_\delta\setminus K_{2\delta})},
\end{gathered}
$$
as required.
\end{proof}
We are now ready to prove Theorem~\ref{l:integrated-xpd}.
Take $\chi_\delta$ from the proof of Proposition~\ref{l:parametrix-estimate}
and consider $u_2=u-\chi_\delta(r) u_1$.
Then
$$
\Box_g u_2=(1-\chi_\delta)f-[\Box_g,\chi_\delta]u_1
$$
is supported in $K_\delta$; moreover, by Proposition~\ref{l:parametrix-estimate},
$$
\|e^{\nu t_+}\Box_g u_2\|_{H^{s+\varkappa}(M_\delta)}
\lesssim \|e^{\nu t_+}f\|_{H^{s+\varkappa}(M_\delta)}.
$$
Therefore, we may apply Proposition~\ref{l:estimate-k-m} to $u_2$ to get
$$
\|e^{\nu t_+}(u_2-\Pi_0(\Box_g u_2))\|_{H^s(M_\delta)}
\lesssim \|e^{\nu t_+}f\|_{H^{s+\varkappa}(M_\delta)}.
$$
Note that $\Pi_0(\Box_g u_2)=\Pi_0f$, as
$$
\Pi_0\Box_g\chi_\delta(r) u_1
=\lim_{T\to +\infty}\int_{\partial (M\cap \{t_+\leq T\})}
g(\nabla(\chi_\delta(r)u_1),\vec n)\,dS.
$$
The integral over the cap $M\cap\{t_+=T\}$ converges to zero, as $u_1$
is exponentially decaying in time. As for the timelike piece $\partial
M\cap \{t_+\leq T\}$, the normal vector $\vec n$ is tangent to
$\partial M$ and we can use this to replace the integral of $g(\nabla
u_1,\vec n)$ over the timelike piece by a certain integral over the
spheres $\partial M\cap \{t_+=T\}$; the latter will decay
exponentially as $T\to +\infty$.  We now get
$$
\begin{gathered}
\|e^{\nu t_+}(u-\Pi_0f)\|_{H^s(M_\delta)}\lesssim
\|e^{\nu t_+}\chi_\delta u_1\|_{H^s(M_\delta)}
+\|e^{\nu t_+}(u_2-\Pi_0(\Box_g u_2))\|_{H^s(M_\delta)}\\\lesssim
\|e^{\nu t_+}f\|_{H^{s+\varkappa}(M_\delta)},
\end{gathered}
$$
which finishes the proof.

\noindent\textbf{Acknowledgements.}
I would like to thank Daniel Tataru for suggesting that the results
of~\cite{skds} can lead to improved energy estimates and for providing
several key ideas in the proof. Thanks also to Maciej Zworski for
helpful discussions and to Kiril Datchev and Andr\'as Vasy for several
enlightening discussions on~\cite{m-sb-v} and~\cite{d-v}. I am also
grateful for partial support from NSF grant DMS-0654436. Finally, I
am thankful to an anonymous referee for careful reading and
suggestions to improve the manuscript.


\end{document}